\documentclass[reqno]{amsart}
\usepackage[bb=dsserif]{mathalpha}
\usepackage{hyperref}
\usepackage{float}
\usepackage{amsmath}
\usepackage{graphicx} 
\usepackage{latexsym}
\usepackage{amsfonts}
\usepackage{amssymb}
\usepackage{xcolor} 
\theoremstyle{plain}
\newtheorem{theorem}{Theorem}
\newtheorem{corollary}[theorem]{Corollary}
\newtheorem{lemma}[theorem]{Lemma}

\theoremstyle{definition}

\newtheorem{remark}[theorem]{Remark}
\newtheorem*{remark*}{Remark}

\newcommand{\pr}{\mathbb{P}}
\newcommand{\R}{\mathbb{R}}

\newcommand{\lf}{\lfloor}
\newcommand{\rf}{\rfloor}
\newcommand{\E}{\mathbb{E}}
\newcommand{\ph}{\varphi}
\newcommand{\Z}{\mathbb{Z}}

\title[Corrected diffusion approximation for conditioned walks]
{Corrected diffusion approximation for random walks conditioned to stay positive}
\date{}
\author[Denisov]{Denis Denisov}
\address{Department of Mathematics, University of Manchester, UK}
\email{denis.denisov@manchester.ac.uk}

\author[Tarasov]{Alexander Tarasov}
\address{Faculty of Mathematics, Bielefeld University, Germany}
\email{atarasov@math.uni-bielefeld.de}

\author[Wachtel]{Vitali Wachtel}
\address{Faculty of Mathematics, Bielefeld University, Germany}
\email{wachtel@math.uni-bielefeld.de}

\begin{document}

\begin{abstract}
    Let $S_n$ be a random walk with i.i.d. increments which have zero mean and finite variance. For every $x\ge0$ we define the stopping time $\tau_x:=\inf\{n\ge1:x+S_n\le0\}$ and consider the probabilities 
    $\pr(x+S_n\ge y,\tau_x>n)$. We study the quality of the normal approximation for these probabilities and derive a Berry-Esseen-type inequality for $\pr(x+S_n\ge y|\tau_x>n)$. Our Theorem~\ref{thm:BE} is an extension of the results in \cite{DTW24a} where we have considered the special case $x=0$. It is also worth mentioning that Theorem~\ref{thm:BE} complements the results of Siegmund and Yuh~\cite{Siegmund-Yuh} on the corrected diffusion approximation.
\end{abstract}

\maketitle
\section{Introduction}
Let  $\{X_k\}$ be a sequence of  independent, identically distributed  random variables with zero mean 
$\E X_1=0$ and finite variance $\E X_1^2=:\sigma^2\in(0,\infty)$. 
Consider a random walk 
 $\{S_n; n\ge0\}$ defined as follows, 
 $S_0=0$ and 
\begin{align*}
    S_n:=X_1+X_2+\ldots+X_n,\ n\ge1.
\end{align*}
For every $x\ge 0$ we define the stopping time 
$$
\tau_x:=\inf\{n\ge1:x+S_n\le 0\}.
$$
The main purpose of the present paper is to study the quality of normal approximation for probabilities $\pr(x+S_n\ge y,\tau_x>n)$ and
$\pr(x+S_n\ge y|\tau_x>n)$. 

In \cite{DTW24a} we have considered the case $x=0$ and have proved that there exists an absolute constant $A_0$ such that 
$$
\sup_{y\ge0}\left|\pr(S_n\ge y|\tau_0>n)-e^{-y^2/2\sigma^2n}\right|
\le A_0\frac{(\E |X_1|^3)^3}{\sigma ^9\sqrt{n}}
$$
and 
$$
\left|\frac{\pr(\tau_0>n)}{\sqrt{\frac{2}{\pi}}\E|S_{\tau_0}|n^{-1/2}}-1\right|
\le A_0\frac{(\E |X_1|^3)^3}{\sigma ^9\sqrt{n}}
$$
for all $n\ge 1$. These estimates can be seen as an analogue of the classical
Berry-Esseen inequality, which says that
\begin{equation}\label{eq.berry-esseen.classical}
    \left|\pr\left(\frac{S_n}{\sigma \sqrt{n}} \le x\right) - \Phi(x)\right|
\le
    \gamma_0\frac{\E |X_1|^3}{\sigma ^3\sqrt{n}},
\end{equation}
where $\Phi$ stands for the standard normal distribution function and  one can take $\gamma_0=0.4785$. 

In the present note we are going to generalize the results from \cite{DTW24a} to the case of arbitrary starting point $x\ge 0$. It turns out that in this case it is more convenient to consider the probabilities $\pr(x+S_n\ge y,\tau_x>n)$ than the conditioned probabilities $\pr(x+S_n\ge y|\tau_x>n)$. Here is our main result.
\begin{theorem}
\label{thm:BE}
Assume $\E X_1 = 0, \E |X_1|^2 = \sigma^2$ and $\E |X_1|^3 < \infty$. 
Then there exists an absolute constant $A_1$ such that
\begin{align}
\label{eq:BE1}
    &\left|\pr(x+S_n\ge y,\tau_x>n)
   -\left[\Phi\left(\frac{y+x}{\sigma\sqrt{n}}\right)
      -\Phi\left(\frac{y-x}{\sigma\sqrt{n}}\right)\right]
   -\frac{2}{\sqrt{2\sigma^2\pi n}} e^{-\frac{y^2}{2\sigma^2n}} \E|x+S_{\tau_x}|\right|  \nonumber \\
  &\hspace{4cm} \le  
  A_1\frac{(\E|X_1|^3)^3\E|S_{\tau_x}|}{\sigma^9 \sqrt{n}(x+\sqrt{n})}.
\end{align}
\end{theorem}
Let $B_t$ denotes the standard Brownian motion. For every fixed $x$ we denote 
$$
\tau^{bm}_x:=\inf\{t>0:x+B_t\le0\}.
$$
Then, by the reflection principle for the Brownian motion,
\begin{align*}
\pr(x+B_t\ge y,\tau_x^{bm}>t)
&=\pr(x+B_t\ge y)-\pr(-x+B_t\ge y)\\
&=\Phi\left(\frac{y+x}{\sqrt{t}}\right)-\Phi\left(\frac{y-x}{\sqrt{t}}\right).
\end{align*}
Thus, the estimate \eqref{eq:BE1} can be seen as a corrected diffusion approximation for random walks conditioned to stay positive, the correction is given by the term $\frac{2}{\sqrt{2\sigma^2\pi n}} e^{-\frac{y^2}{2\sigma^2n}} \E|x+S_{\tau_x}|$. This term is bigger than the right hand side in \eqref{eq:BE1} in the case when $x=o(n^{1/2})$. To illustrate the effect of the correction we consider the probability $\pr(\tau_x>n)$. Putting $y=0$ in \eqref{eq:BE1}, we obtain 
\begin{align*}
&\left|\pr(\tau_x>n)-\left[\Phi\left(\frac{x}{\sigma\sqrt{n}}\right)
      -\Phi\left(\frac{-x}{\sigma\sqrt{n}}\right)\right]
         -\frac{2}{\sqrt{2\sigma^2\pi n}} \E|x+S_{\tau_x}|\right|\\
&\hspace{4cm}         
         \le  A_1\frac{(\E|X_1|^3)^3\E|S_{\tau_x}|}{\sigma^9 n}.
\end{align*}
If we assume now that $x=x_n\to\infty$ and $x_n=o(n^{1/2})$ then,
using the fact that $\lim_{x\to\infty}\E|x+S_{\tau_x}|=:E\in(0,\infty)$ under the assumption $\E|X_1|^3 < \infty$,
we conclude that 
$$
\left|\pr(\tau_x>n)-\left[\Phi\left(\frac{x}{\sigma\sqrt{n}}\right)
      -\Phi\left(\frac{-x}{\sigma\sqrt{n}}\right)\right]
         -\frac{2E}{\sqrt{2\sigma^2\pi n}} \right|
    =o(n^{-1/2}).
$$
This simple argument shows that Theorem~\ref{thm:BE} gives, for $x=o(\sqrt{n})$, a better result than the bound
\begin{align}
\label{eq:Ales-bound}
\sup_{x\ge 0}\left|\pr(\tau_x>n)-\left[\Phi\left(\frac{x}{\sigma\sqrt{n}}\right)
      -\Phi\left(\frac{-x}{\sigma\sqrt{n}}\right)\right]\right|
\le A\frac{\E|X_1|^3}{n^{1/2}},
\end{align}
which has been obtained in \cite{Aleshkyavichene73}. If $x=O(n^{3/4})$ then the result in Theorem~\ref{thm:BE} gives a better rate of convergence than the rate obtained very recently by Grama and Xiao~\cite{GX24}.

We notice also that our Theorem~\ref{thm:BE} can be seen as a complement to the corrected diffusion approximation obtained by Siegmund and Yuh in \cite{Siegmund-Yuh} in the case when $x=a\sqrt{n}$ with some $a>0$. They have obtained a short asymptotic expansions for the probability $\pr(x+S_n\ge y,\tau_x\le n)$, which can be transferred, by using classical expansions in the CLT, into expansions for $\pr(x+S_n\ge y,\tau_x>n)$.

It is well-known that if $x=o(\sqrt{n})$ then the distribution of $\frac{x+S_n}{\sigma\sqrt{n}}$ conditioned on $\tau_x>n$ converges towards the Rayleigh distribution. Theorem~\ref{thm:BE} allows one to obtain a rate of convergence in this limit theorem.
\begin{corollary} \label{cor}
Under the assumptions of Theorem~\ref{thm:BE}, there exist absolute constants $A_2, A_3$ such that, for all $x\le\sqrt{n}$,
\begin{align}\label{eq:BE2}
    \left|
        \pr(x+S_n\ge y|\tau_x>n)
   -
        e^{-\frac{y^2}{2\sigma^2n}}
    \right|  
\le
    A_2\frac{(\E|X_1|^3)^3}{\sigma^9\sqrt{n}}
     + A_3\frac{x^2}{\sigma^2 n}
\end{align}
and 
\begin{align}\label{eq:BE3}
    \left|
        \frac{\pr(\tau_x>n)}
        {\sqrt{\frac{2}{\sigma^2\pi}}\E|S_{\tau_x}|n^{-1/2}}
    -
        1
    \right|
\le
    A_2\frac{(\E|X_1|^3)^3}{\sigma^9\sqrt{n}}
     + A_3\frac{x^2}{\sigma^2 n}.
\end{align}
\end{corollary}
This corollary implies uniform rate of convergence of order $n^{-1/2}$ when $x\le n^{1/4}$. The term
$\frac{x^2}{n}$  on the right hand sides of \eqref{eq:BE2} and \eqref{eq:BE3} is caused by the approximation of
$\Phi\left(\frac{y+x}{\sigma\sqrt{n}}\right)
-\Phi\left(\frac{y-x}{\sigma\sqrt{n}}\right)$ by
$\frac{2x}{\sqrt{2\sigma^2\pi n}}e^{-\frac{y^2}{2\sigma^2n}}$ and appears naturally in the asymptotic expansions for 
$\frac{\Phi\left(\frac{y+x}{\sigma\sqrt{n}}\right)
-\Phi\left(\frac{y-x}{\sigma\sqrt{n}}\right)}
{\Phi\left(\frac{x}{\sigma\sqrt{n}}\right)
-\Phi\left(\frac{-x}{\sigma\sqrt{n}}\right)}$ in the case when $x=o(\sqrt{n})$.

Comparing \eqref{eq:BE1} with the classical Berry-Esseen inequality \eqref{eq.berry-esseen.classical} and with Aleshkya-\\ vichene's
bound~\eqref{eq:Ales-bound}, we see that the right hand side of \eqref{eq:BE1} contains the third power of the Lyapunov ratio $\frac{\E|X_1|}{\sigma^3}$. We believe that the optimal bound should be linear in the Lyapunov ratio and that the third power is caused by our approach which uses the Berry-Esseen inequality to bound local probabilities, see \eqref{eq.conc} below. We next show that imposing some structural properties on the distribution of increments $\{X_k\}$
allows one to improve the bound in Theorem~\ref{thm:BE}.

%\textcolor{red}{TODO: Organisation of the paper.}
%%%%%%%%%%%%%%%%%%%%%%%%%%%%%%%%%%%%%%%%%%%%%%%%%%%%%%%%%%%%%%%%%%%%%%%%%%%%%%%%%%%%%%%%%%%%%%%%%%%%%%%%%%%%%%%%%%%%%%%%%%%%%%%%%%%%%%%%%%%%%%%%%%%%%%%%%%%%%%%%%%
\section{Preliminary estimates}
In what follows we shall assume, without loss of generality, that
$$
\sigma^2=1.
$$

We first extend upper bounds obtained in \cite{DTW24a} to the case of arbitrary starting point $x$.
Using the classical Berry-Esseen inequality~\eqref{eq.berry-esseen.classical}, we conclude that, uniformly in $x$ and $y$, 
\begin{align}
\nonumber 
  \pr(x+S_n\in [y,y+z]) 
  &\le 2\gamma_0\frac{\E|X_1|^3}{\sqrt n}
+\frac{1}{\sqrt{2\pi}}
\int_{(y-x)/\sqrt n}^{(y+z-x)/\sqrt n} e^{-z^2/2}dz 
\\\label{eq.conc}
&\le \frac{\gamma_1(z)}{\sqrt{2n}}, 
\end{align}
where $\gamma_1(z):=\sqrt 2 \E|X_1|^3+z \pi^{-1/2}$. 
This implies that 
\begin{align}
\label{sn.tau}
\nonumber
&\pr(x+S_n\in [y,y+z],\tau_x>n)\\
\nonumber
&\hspace{1cm}\le\int_0^\infty
\pr(x+S_{\lfloor n/2\rfloor}\in dw,\tau_x>n/2)
\pr(S_{n-\lfloor n/2\rfloor}\in [y-w,y-w+z])\\
&\hspace{1cm}\le
\frac{\gamma_1(z)}{\sqrt{n}}\pr(\tau_x>n/2).
\end{align}

\begin{lemma}\label{lem:tau-tail}
For all $n \ge  8(\E|X_1|^3)^2$ one has
\begin{equation}
\label{eq.tau2}
    \pr(\tau_x>n) 
\le
    6\frac{x+|\E[x+S_{\tau_x}]|}{x+\sqrt{n}}
=
    6\frac{\E|S_{\tau_x}|}{x+\sqrt{n}}.
\end{equation}

\end{lemma}
\begin{proof}
Applying Lemma 25 in~\cite{DSW18} to the stopping time $\tau_x$, we conclude that
$$
\pr(\tau_x>n) \le \frac{\E[x+S_n;\tau_x>n]}{\E[(x+S_n)^+]}
$$
for every $x\ge0$. (As usual, $z^+$ denotes the positive part of $z$, i.e.
$z^+=\max\{z,0\}$.)

Applying the optional stopping theorem to the martingale $x+S_n$ with the stopping time $\tau_x\wedge n$, we infer that
$$
x = \E[x + S_n, \tau_x > n] + \E [x+S_{\tau_x}, \tau_x \le n]
$$
and hence
$$
    \E[x+S_n;\tau_x>n]
=
    x-\E[x+S_{\tau_x};\tau_x\le n]
\le 
    x+|\E[x+S_{\tau_x}]|.
$$
Consequently,
\begin{equation}\label{eq.tau}
    \pr(\tau_x>n) 
\le
    \frac{x+|\E[x+S_{\tau_x}]|}{\E[(x+S_n)^+]}.
\end{equation}
Using the classical Berry-Esseen bound \eqref{eq.berry-esseen.classical},
we obtain
\begin{align*}
    \E[(x+S_n)^+]
&=
    \int_0^\infty 
        \pr(x+S_n>y) dy
=
    \int_{-x}^\infty
        \pr(S_n>y) dy\\ 
&\ge x\pr(S_n>0)+\int_0^{\sqrt{n}}\pr(S_n>y)dy\\
&= x\pr(S_n>0)+
    \sqrt n
    \int_0^{1} 
    \pr(S_n>y\sqrt n) dy\\
&\ge x\left(\frac{1}{2}-\gamma_0\frac{\E|X_1|^3}{\sqrt{n}}\right)+
    \sqrt{n}
    \int_0^{1} 
        \overline \Phi(y)dy 
        -\gamma_0 \E|X_1|^3. 
\end{align*}
Recalling that $\gamma_0\le 0.4785$ and noticing that
$\int_0^{1}\overline \Phi(y)dy\ge 0.341$, we conclude that if
$n\ge 8(\E|X_1|^3)^2$ then 
$$
\E[(x+S_n)^+]\ge\frac{1}{6}(x+\sqrt{n}).
$$
This completes the proof of the lemma.
\end{proof}
\begin{lemma}
\label{lem:lorden}
There exists an absolute constant $C$ such that
\begin{align}\label{eq:2nd-moment}
    \E |u+S_{\tau_u}|^2
\le
    C \E|S_{\tau_u}| (\E |X_1|^3)^2,\quad u\ge0.
\end{align}
Furthermore,
\begin{align}
\label{eq:Mogul}
    \E |u+S_{\tau_u}| \le 3 \frac{\E |X_1|^3}{\E |X_1|^2}, \quad u\ge0.
\end{align}
\end{lemma}
\begin{proof}
The second inequality is proved by Mogulskii, see Theorem 2 in \cite{Mogulskii74}.

To prove the first bound we define
\begin{align*}
    \phi(x) 
:=
    \E 
    \Big[
        \sum_{k=0}^{\tau_0-1}
            \mathbb{1}_{\{S_k\in[x,x+1)\}}\Big]
               =\sum_{k=0}^\infty\pr(S_k\in[x,x+1),\tau_0>k).
\end{align*}
By (39) in \cite{VW09},
$$
\phi(x)=H(x+1)-H(x),
$$
where $H(x)$ denotes the renewal function corresponding to strict ascending ladder heights.  
By subadditivity of renewal function, $H(x+1)-H(x)\le H(1)$ for all $x$. Therefore,
$$
\sup_{x \in \R}\phi(x)=\phi(0)\le H(1).
$$
In \cite[Corollary 8]{DTW24a} we have shown that 
$$
H(x)\le 2\E|S_{\tau_0}|(x+c_2\E|X_1|^3),\quad x\ge0,
$$
where $c_2$ is an absolute constant.
Consequently,
\begin{equation}
\label{eq:phi-max} 
\sup_{x \in \R}\phi(x)\le 2(1+c_2)\E|X_1|^3\E|S_{\tau_0}|.
\end{equation}
Splitting the path of the walk into independent cycles by descending ladder epochs, we obtain
\begin{align*}
    &\sum_{k = 0}^\infty
        \pr( u + S_k \in [w,w+1), \tau_u > k)\\
&\hspace{1cm}=
    \E 
    \Big[ 
        \sum_{k=0}^{\tau_u-1}
            \mathbb{1}_{\{u + S_k \in [w,w+1)\}}
    \Big]
\\
&\hspace{1cm}=
    \phi(w-u)
+
    \E
    \Big[
        \sum_{j=1}^{\theta(u)-1}
            \phi(w-u + \chi_1^- + \dots + \chi_j^-)
    \Big],
\end{align*}
where $\theta(u):=\inf\{j\ge1: \chi^-_1+\ldots+\chi^-_j\ge u\}$ and $\chi_i^-$ is the $i$-th weak descending ladder height. Therefore, $\chi_i^-$ are independent copies of $|S_{\tau_0}|$. Combining this representation with \eqref{eq:phi-max} and noting that, due to the Wald identity, 
$$
\E|S_{\tau_0}|\E\theta(u)=\E|S_{\tau_u}|,
$$
we conclude that 
\begin{align} \label{eq:staumomentssetim}
    \sum_{k =0}^\infty 
        \pr( u + S_k \in [w,w+1), \tau_u > k)
\le
    \max_{x \in \R}\phi(x)\E \theta(u)
\le 
    c_3\E|X_1|^3\E|S_{\tau_u}|,
\end{align}
where $c_3$ is an absolute constant.

By the total probability law,
\begin{align*}
    \E |u+S_{\tau_u}|^{2}
&=
    \sum_{k=0}^{\infty}
    \int_0^\infty \pr(u+S_k\in dw,\tau_u>k)\E[(X+w)^2;X\le-w]\\
&\le
    \sum_{w=0}^\infty\sum_{k=0}^\infty 
    \pr(u+S_k\in [w,w+1),\tau_u>k)\E[X^2;X\le-w].
%     \\
% &=
%     \sum_{w=0}^\infty \phi(w)\E[X^2;X\le-w]. 
\end{align*}
Applying now \eqref{eq:staumomentssetim} we conclude that
\begin{align*}
    \E |u+S_{\tau_u}|^2
&\le
    c_3\E|X_1|^3\E|S_{\tau_u}|
    \sum_{w=0}^\infty \E[X^2;X\le-w]\\
&\le c_3\E|X_1|^3\E|S_{\tau_u}|
    \sum_{w=0}^\infty \E[X^2;|X|\ge w]\\
&\le c_3\E|X_1|^3\E|S_{\tau_u}|\E[X^2(|X|+1)]
\le 2c_3 (\E|X_1|^3)^2\E|S_{\tau_u}|, 
\end{align*}
using that $\E|X|^3\ge 1$ because of $\E[X^2]=1$. 
This completes the proof of the lemma.
\end{proof}

The next lemma provides an upper bound for conditional local probabilities and is the main difference to the approach used in \cite{DTW24a}. Lemma 7 there gives a similar bound for the particular case $x=0$ and is based on a representation for the local probabilities which follows from the Wiener-Hopf factorisation. Since the factorisation is not directly applicable to positive starting point $x$, we use a different, even simpler, approach based on the time reversal.

\begin{lemma}\label{lem:concrete.bound.alternative}
For all $n\ge  32 (\E|X_1|^3)+4$ one has 
\begin{align*}
    \pr(x+S_n \in [y,y+1), \tau_x > n)
    &\le
    288\gamma_1(1)
    \frac{\E|S_{\tau_x}|\left(y+4\E|X_1|^3\right)}{n(x+\sqrt{n})}\\
&\le
    288\gamma_1(1)
    \frac{\left(x+3\E|X_1|^3\right)\left(y+4\E|X_1|^3\right)}{n(x+\sqrt{n})}.
\end{align*}
\end{lemma}
\begin{proof}
Consider a random walk $\{S'_n\}_{n\ge 0}\overset{d}{=} \{- S_n\}_{n\ge 0}$. %such that $S'_n \overset{d}{=} - S_n$ for every $n$. 
Define 
\[
\tau'_y = \inf \{n:y+S'_n \le 0\}
\] and let $m = \lf n/2 \rf$.
Due to Lemma~\ref{lem:tau-tail}, for all $n \ge 16 (\E|X_1|^3)^2 + 2$ and all $x,y\ge0$, 
we have the bounds
\begin{align} \label{eq:tau2:corr0}
    \pr(\tau_x > m) 
\le
    6\frac{\E|S_{\tau_x}|}{x+\sqrt{m}}
\le
    12\frac{\E|S_{\tau_x}|}{x+\sqrt{n}} 
\end{align}
and
\begin{align}\label{eq:tau2:corr'0}
    \pr(\tau'_y > m) 
\le
    12\frac{\E|S'_{\tau'_y}|}{\sqrt{n}}. 
\end{align}
Applying the Markov property at time $m$, we obtain
\begin{multline*}
    \pr(y + S'_n \in [z,z+1), \tau'_y > n)\\
\le
    \int_0^\infty
        \pr(y+S'_m \in du, \tau'_y > m)
        \pr(S'_{n-m} \in [z-u,z-u+1) ).
\end{multline*}
Noting that $X_1'$ has zero mean and unit variance and $\E|X_1'|^3 = \E |X_1|^3$, we infer that \eqref{eq.conc} can be applied to
$\pr(S'_{n-m} \in [z-u,z-u+1) )$. As a result we have
\begin{align*}
    \pr(y + S'_n \in [z,z+1), \tau'_y > n)
&\le
    \pr(\tau'_y > m)\frac{\gamma_1(1)}{\sqrt{2(n-m)}}\\
&\le
    \pr(\tau'_y > m)\frac{\gamma_1(1)}{\sqrt{n}},
\end{align*}
where the last inequality holds since $2(n-m) \ge n$. Applying now \eqref{eq:tau2:corr'0}, one gets
\begin{align} \label{eq:conc.bound.weak0}
    \pr(y + S'_n \in [z,z+1), \tau'_y > n)
\le
    12\gamma_1(1)\frac{\E|S'_{\tau'_y}|}{n}.
\end{align}
Applying the same argument to the walk $\{S_n\}$ and using \eqref{eq:tau2:corr0} instead of \eqref{eq:tau2:corr'0}, we obtain 
\begin{align} \label{eq:conc.bound.weak00}
    \pr(x + S_n \in [z,z+1), \tau_x > n)
\le
    12\gamma_1(1)\frac{\E|S_{\tau_x}|}{\sqrt{n}(x+\sqrt{n})}.
\end{align}
Using once again the Markov property, we have
\begin{multline} \label{eq:lem:conc.bound.splitting0}
    \pr(x+S_n \in [y,y+1), \tau_x > n)\\
=
    \int_0^\infty
        \pr(x+S_m \in dz, \tau_x > m)
        \pr(z+S_{n-m} \in [y,y+1), \tau_z > n-m ).
\end{multline}
Reverting the time, one gets easily
\begin{align*}
    \pr(z+S_{k} \in [y,y+1), \tau_z > k ) 
\le
    \pr(y+1+S'_{k} \in [z,z+1), \tau'_{y+1} > k ),\quad k\ge1.
\end{align*}
Combining this bound with \eqref{eq:conc.bound.weak0}, we conclude that
\begin{align*}
    \pr(z+S_{n-m} \in [y,y+1), \tau_z > n-m )
\le &
    12\gamma_1(1)\frac{\E|S'_{\tau'_{y+1}}|}{n-m}\\
\le &
    24\gamma_1(1)\frac{\E|S'_{\tau'_{y+1}}|}{n}
\end{align*}
 for all $n\ge 32 (\E|X_1|^3)^2 + 4$.
Substituting this into \eqref{eq:lem:conc.bound.splitting0} and applying \eqref{eq:tau2:corr0}, we finally obtain
\begin{align*}
    \pr(x+S_n \in [y,y+1), \tau_x > n)
\le &
    24\gamma_1(1)\frac{\E|S'_{\tau'_{y+1}}|}{n}\pr(\tau_x > m)\\
\le &
    288\gamma_1(1)\frac{\E|S_{\tau_x}|\E|S'_{\tau'_{y+1}}|}{n(x+\sqrt{n})}.
\end{align*}
Using \eqref{eq:Mogul} and recalling that $\sigma^2=1$, we get 
\begin{align*}
    \E|S_{\tau_x}|\E|S'_{\tau'_{y+1}}|
    &\le \E|S_{\tau_x}|\left(y+1+3\E|X_1|^3\right)\\
    &\le \left(x+3\E|X_1|^3\right)\left(y+4\E|X_1|^3\right).
\end{align*}
This completes the proof of the lemma.
\end{proof}
\begin{lemma}
\label{lem:S_tau}
There exists an absolute constant $C_1$ such that, for all $k \ge 32 (\E|X_1|^3)^2+5$, one has the bounds
\begin{align*}
&   \pr(\tau_x=k)
\le 
    C_1\frac{\E|S_{\tau_x}|}{k(x+\sqrt{k})}(\E|X_1|^3)^2
% \le 
%     C_1\frac{x+3\E|X_1|^3}{k(x+\sqrt{k})}(\E|X_1|^3)^2
    ,\\
&   \E[|x+S_{\tau_x}|;\tau_x=k]
\le 
    C_1\frac{\E|S_{\tau_x}|}{k(x+\sqrt{k})}(\E|X_1|^3)^2
% \le 
%     C_1\frac{x+3\E|X_1|^3}{k(x+\sqrt{k})}(\E|X_1|^3)^2
    ,
\end{align*}
$$
    \E[\gamma_1(|x+S_{\tau_x}|);\tau_x=k]
\le 
    C_1\frac{\E|S_{\tau_x}|}{k(x+\sqrt{k})}(\E|X_1|^3)^3
% \le 
%     C_1\frac{x+3\E|X_1|^3}{k(x+\sqrt{k})}(\E|X_1|^3)^3
    .
$$
and
$$
    \E[|x+S_{\tau_x}|^2;\tau_x=k]
\le 
    C_1\frac{\E|S_{\tau_x}|}{\sqrt{k}(x+\sqrt{k})}(\E|X_1|^3)^2
% \le 
    % C_1\frac{x+3\E|X_1|^3}{\sqrt{k}(x+\sqrt{k})}(\E|X_1|^3)^2
    .
$$
\end{lemma}
\begin{proof}
Fix some $a,b\ge 0$ and consider the expected value $\E[(a|x+S_{\tau_x}|+b);\tau=k]$. By the total probability law,
\begin{align*}
&\E[(a|x+S_{\tau_x}|+b);\tau_x=k]\\
&\hspace{1cm}
\le \int_0^\infty \pr(x+S_{k-1}\in dy,\tau_x>k-1)\E[(-aX+b);X\le-y]\\
&\hspace{1cm}
\le\sum_{j=0}^\infty \pr(x+S_{k-1}\in [j,j+1),\tau_x>k-1] \E[(-aX+b);X\le-j]).
\end{align*}
Applying now Lemma \ref{lem:concrete.bound.alternative} with $n=k-1$, we get 
\begin{align*}
&\E[(a|x+S_{\tau_x}|+b);\tau_x=k]\\
&\hspace{1cm}\le 288\gamma_1(1)\frac{\E|S_{\tau_x}|}{(k-1)(x+\sqrt{k-1})}
\sum_{j=0}^\infty \left(j+4\E|X_1|^3\right)\E[(-aX+b);X\le-j])\\
&\hspace{1cm}
\le 815\gamma_1(1)\frac{\E|S_{\tau_x}|}{k(x+\sqrt{k})}
\E\left[(-aX+b)\sum_{j\in[0,-X]}\left(j+4\E|X_1|^3\right);X\le 0\right]\\
&\hspace{1cm}
\le 815\gamma_1(1)\frac{\E|S_{\tau_x}|}{k(x+\sqrt{k})}
\E\left[(-aX+b)(-X+1)\left(\frac{-X}{2}+4\E|X_1|^3\right);X\le 0\right].
\end{align*}
Taking here $a=0$ and $b=1$, we get 
$$
\pr(\tau_x=k)\le 815\gamma_1(1)\frac{\E|S_{\tau_x}|}{k(x+\sqrt{k})}
\E\left[(-X+1)\left(\frac{-X}{2}+4\E|X_1|^3\right);X\le0\right]
$$
Using next the Jensen inequality and noting that $\sigma^2=1$ implies that $\E|X_1|^3\ge1$, we conclude that 
$$
\pr(\tau_x=k)\le C_1\frac{\E|S_{\tau_x}|}{k(x+\sqrt{k})}(\E|X_1|^3)^2.
$$
Choosing $a=1$ and $b=0$ and applying once again the Jensen inequality, we infer that the same inequality is valid for $\E[|x+S_{\tau_x}|;\tau=k]$.
Further, taking $a=\E|X_1|^3$ and $b=\pi^{-1/2}$, we obtain the third claim.

Finally,
\begin{align*}
 &\E[|x+S_{\tau_x}|^2;\tau_x=k]\\
 &\hspace{1cm}
 =\int_0^\infty\pr(x+S_{k-1}\in dy,\tau_x>k-1)\E[(-X+y)^2;X\le-y]\\
 &\hspace{1cm}
 \le\sum_{j=0}^{\infty}\pr(x+S_{k-1}\in [j,j+1),\tau_x>k-1)\E[(-X+j)^2;X\le-j].
\end{align*}
Applying now \eqref{eq:conc.bound.weak00}, we get 
\begin{align*}
    \E[|x+S_{\tau_x}|^2;\tau_x=k]
&\le 
    12\gamma_1(1)\frac{\E|S_{\tau_x}|}{\sqrt{k-1}(x+\sqrt{k-1})}
    \sum_{j=0}^\infty\E[(-X+j)^2;X\le-j]\\
&\le
    24\gamma_1(1)\frac{\E|S_{\tau_x}|}{\sqrt{k}(x+\sqrt{k})}\E|X_1|^3.
\end{align*}
This finishes the proof of the lemma.
\end{proof}
\begin{lemma}\label{lem:tau-exp}
For all $n \ge  8(\E|X_1|^3)^2$ one has
\begin{equation}
\label{eq:tr.exp}
\sum_{k=1}^n k\pr(\tau_x=k)
\le\E[\tau_x\wedge n]\le C\E|X_1|^3\frac{n\E|S_{\tau_x}|}{x+\sqrt{n}}.
\end{equation}
\end{lemma}
\begin{proof}
Let us first show that the desired inequality holds in the case $x\ge\sqrt{n}$. Indeed, in this case, noting that $\E|S_{\tau_x}|\ge x$, we have
$$
\E[\tau_x\wedge n]\le n\le n\frac{\E|S_{\tau_x}|}{x}
\le 2n\frac{\E|S_{\tau_x}|}{x+\sqrt{n}}.
$$
Noting that $\sigma^2=1$ implies that $\E|X_1|^3>1$, we conclude that
if $x\ge\sqrt{n}$ then \eqref{eq:tr.exp} holds with $C=2$.

We now assume that $x\ge \sqrt{n}$. Applying the optional stopping theorem to the martingale $S_n^2-n$, we get
\begin{align*}
\E[\tau_x\wedge n]
=\E[S^2_{\tau_x\wedge n}]
    &=\E[S_{\tau_x}^2;\tau_x\le n]+\E[S_n^2;\tau_x>n]\\
    &\le \E[S_{\tau_x}^2]+\E[S_n^2;\tau_x>n]\\
    &\le 2x^2+2\E[|x+S_{\tau_x}|^2]+\E[S_n^2;\tau_x>n].
\end{align*}
By the assumptions $x\le \sqrt{n}$, $n\ge 8(\E|X_1|^3)^2$ and by Lemma~\ref{lem:lorden},
\begin{align*}
2x^2+2\E[|x+S_{\tau_x}|^2]
&\le 4n\frac{\E|S_{\tau_x}|}{x+\sqrt{n}} 
+C_0(\E|X_1|^3)^2\E|S_{\tau_x}|\\
&\le C_1\E|X_1|^3 n\frac{\E|S_{\tau_x}|}{x+\sqrt{n}}.
\end{align*}
Thus, it remains to show that
\begin{align}
 \label{eq:nl1}
 \E[S_n^2;\tau_x>n]\le 
 C\E|X_1|^3 n\frac{\E|S_{\tau_x}|}{x+\sqrt{n}}+\frac{1}{2}\E[\tau_x\wedge n].
\end{align}
To this end we first estimate the probability $\pr(S_n\ge y,\tau_x>n)$. We shall follow the strategy of the proof of Lemma~1.2 in   Doney and Jones~\cite{JD2012} and define two auxiliary stopping times
$$
T_y:=\inf\{k\ge1: S_k\ge y/2\}
\quad\text{and}\quad 
\eta_y:=\inf\{k\ge1: X_k\ge y/4\}.
$$
Then we have
\begin{align}
\label{eq:pbound1}
\nonumber
\pr(S_n\ge y,\tau_x>n)
&= \pr(S_n\ge y,\eta_y> n,\tau_x>n)+\pr(S_n\ge y,\eta_y\le n,\tau_x>n)\\
&\le \pr(S_n\ge y,\eta_y> n,\tau_x>n)+\pr(\eta_y\le n,\tau_x>n).
\end{align}
We also have
\begin{align}
\label{eq:pbound2}    
\pr(\eta_y\le n,\tau_x>n)
\le \sum_{k=1}^n\pr(\tau_x>k-1)\pr(X>y/4)=\E[\tau_x\wedge n]\pr(X>y/4).
\end{align}
Noting that $S_{T_y}< 3y/4$ on the event $\{\eta_y>n\}$, we obtain 
\begin{align*}
 \pr(S_n\ge y,\eta_y> n,\tau_x>n)
 &\le \sum_{k=0}^{n-1}\pr(\tau_x>k,T_y=k)\pr(S_{n-k}>y/4)\\
 &\le \max_{k\le n}\pr(S_{n-k}>y/4)\sum_{k=0}^{n-1}\pr(\tau_x>k,T_y=k).
\end{align*}
If $y\ge \sqrt{8n}$ then, due to the Doob inequality,
\begin{align*}
\pr(\tau_x>n)&\ge \sum_{k=0}^{n-1}\pr(\tau_x>k,T_y=k)
\pr\left(\min_{j\le n-k}S_j\le-y/2\right)\\
&\ge \frac{1}{2}\sum_{k=0}^{n-1}\pr(\tau_x>k,T_y=k).
\end{align*}
Consequently,
\begin{align*}
 \pr(S_n\ge y,\eta_y> n,\tau_x>n)
 \le 2\pr(\tau_x>n)\pr\left(\max_{k\le n}S_{k}\ge y/4\right).
\end{align*}
Using once again the Doob inequality, we conclude that 
\begin{align}
 \label{eq:pbound3}
\pr(S_n\ge y,\eta_y> n,\tau_x>n)
 \le 2^7\pr(\tau_x>n)\frac{\E|S_n|^3}{y^3}.
\end{align}
Plugging \eqref{eq:pbound2} and \eqref{eq:pbound3} into \eqref{eq:pbound1}, we conclude that 
\begin{align*}
 \pr(S_n\ge y,\tau_x>n)
 \le 2^7\pr(\tau_x>n)\frac{\E|S_n|^3}{y^3}+\E[\tau_x\wedge n]\pr(X>y/4)
\end{align*}
for all $y\ge \sqrt{8n}$. This implies that 
\begin{align*}
\E[S_n^2;\tau_x>n]
&\le 8n\pr(\tau_x>n)+\E[S_n^2;S_n>\sqrt{8n},\tau_x>n]\\
&\le C\pr(\tau_x>n)\left(n+\frac{\E|S_n|^3}{\sqrt{n}}\right)
+\E[\tau_x\wedge n]\E[X^2;X>\sqrt{n/2}].
\end{align*}
According to Theorem~2 in \cite{NP77},
\begin{align}
\label{eq:3-moment}
\E|S_n|^3\le C(\E|X_1|^3n+n^{3/2}).
\end{align}
Combining this inequality with \eqref{eq.tau2}
and noting that, by the Markov inequality,
$$
\E\left[X^2;X>\sqrt{n/2}\right]\le\sqrt{\frac{2}{n}}\E|X_1|^3\le \frac{1}{2}
$$
for $n\ge 8(\E|X_1|^3)^2$, we conclude that \eqref{eq:nl1} holds.
Thus, the proof is complete.
\end{proof}

Besides the bound for the truncated expectation of $\tau_x$, we shall need estimates for some truncated moments of $x+S_{\tau_x}$, which will be proved in subsequent lemmata.
\begin{lemma}
\label{lem:S_tau-1}
For all $n$ and $x$ we have
$$
\E[|x+S_{\tau_x}|;\tau_x\le n]
\le 8\frac{\sqrt{n}}{x+\sqrt{n}}\E[|S_{\tau_x}|].
$$
\end{lemma}
\begin{proof}
If $x\le \sqrt{n}$ then the inequality is immediate from 
$$
\E[|x+S_{\tau_x}|;\tau_x\le n]\le \E[|x+S_{\tau_x}|]\le \E[|S_{\tau_x}|].
$$
In the case $x\ge\sqrt{n}$, using the Markov and Doob inequalities, we have 
\begin{align*}
\E[|x+S_{\tau_x}|;\tau_x\le n]
&\le \E[|S_{\tau_x}|;\tau_x\le n]\\
&\le \E\left[\max_{k\le n}|S_k|;\max_{k\le n}|S_k|\ge x\right]\\
&\le \frac{1}{x}\E\left[\max_{k\le n}|S_k|^2\right]
\le\frac{4}{x}\E[S_n^2]=\frac{4n}{x}\le 8\frac{x\sqrt{n}}{x+\sqrt{n}}.
\end{align*}
Combining this with the observation $|S_{\tau_x}|\ge x$, we complete the proof.
\end{proof}
\begin{lemma}
\label{lem:S_tau-2}
There exists an absolute constant $C$ such that
$$
\sum_{k=1}^n k\E[|x+S_{\tau_x}|;\tau_x=k]
\le C (\E|X_1|^3)^2\frac{n}{x+\sqrt{n}}\E[|S_{\tau_x}|]
$$
for all $n\ge 32(\E|X_1|^3)^2+5$.
\end{lemma}
\begin{proof}
We start by noting that 
\begin{align*}
&\sum_{k=1}^{\lf 32(\E|X_1|^3)^2+5 \rf} k\E[|x+S_{\tau_x}|;\tau_x=k]\\
&\hspace{2cm}
\le (32(\E|X_1|^3)^2+5)\E[|x+S_{\tau_x}|;\tau_x\le 32(\E|X_1|^3)^2+5]\\
&\hspace{2cm}
\le 37(\E|X_1|^3)^2 \E[|x+S_{\tau_x}|;\tau_x\le n].
\end{align*}
Taking into account Lemma~\ref{lem:S_tau-1}, we conclude that
\begin{align}
 \label{eq:S_tau-1.1}
 \nonumber
 \sum_{k=1}^{\lf 32(\E|X_1|^3)^2+5 \rf} k\E[|x+S_{\tau_x}|;\tau_x=k]
 &\le C_1 (\E|X_1|^3)^2 \frac{\sqrt{n}}{x+\sqrt{n}}\E[|S_{\tau_x}|]\\
 &\le C_1\E|X_1|^3\frac{n}{x+\sqrt{n}}\E[|S_{\tau_x}|]
\end{align}
for all $n\ge 32(\E|X_1|^3)^2+5$.
Furthermore, the second bound in Lemma~\ref{lem:S_tau} leads to 
\begin{align*}
\sum_{\lf 32(\E|X_1|^3)^2+5\rf +1}^n k\E[|x+S_{\tau_x}|;\tau_x=k]
&\le C_2(\E|X_1|^3)^2\E[|S_{\tau_x}|]\sum_{k=1}^n\frac{1}{x+\sqrt{k}}\\
&\le C_3(\E|X_1|^3)^2\E[|S_{\tau_x}|]\frac{n}{x+\sqrt{n}}.
\end{align*}
Combining this with \eqref{eq:S_tau-1.1}, we get the desired estimate. 
\end{proof}
\begin{lemma}
\label{lem:S_tau-3}    
There exists an absolute constant $C$ such that 
$$
\E[|x+S_{\tau_x}|^2;\tau_x\le n]
\le C(\E|X_1|^3)^2\frac{\sqrt{n}}{x+\sqrt{n}}\E[|S_{\tau_x}|]
$$
for all $n\ge 32(\E|X_1|^3)^2+5$.
\end{lemma}
\begin{proof}
In the case $x\le\sqrt{n}$ we apply \eqref{eq:2nd-moment} to get 
\begin{align*}
\E[|x+S_{\tau_x}|^2;\tau_x\le n]
&\le \E|x+S_{\tau_x}|^2
\le C(\E|X_1|^3)^2\E|S_{\tau_x}|\\
&\le C_1(\E|X_1|^3)^2\frac{\sqrt{n}}{x+\sqrt{n}}\E|S_{\tau_x}|.
\end{align*}
Assume now that $x>\sqrt{n}$.
Similar to the proof of the Lemma~\ref{lem:S_tau-1},
\begin{align*}
\E[|x+S_{\tau_x}|^2;\tau_x\le 32(\E|X_1|^3)^2+5]
&\le\E\left[\max_{k\le 32(\E|X_1|^3)^2+5}|S_k|^2\right]\\
&\le 148(\E|X_1|^3)^2
\le 296(\E|X_1|^3)^2\frac{1}{x+\sqrt{n}}\E|S_{\tau_x}|\\
&\le 296\E|X_1|^3\frac{\sqrt{n}}{x+\sqrt{n}}\E|S_{\tau_x}|.
\end{align*}
Furthermore, using the last bound in Lemma~\ref{lem:S_tau}, we get 
\begin{align*}
\E[|x+S_{\tau_x}|^2;\tau_x\in( 32(\E|X_1|^3)^2+5,n]]
&\le 
C_1(\E|X_1|^3)^2\E|S_{\tau_x}|\sum_{k=1}^n\frac{1}{\sqrt{k}(x+\sqrt{k})}\\
&\le 
C_1(\E|X_1|^3)^2\E|S_{\tau_x}|\frac{1}{x}\sum_{k=1}^n\frac{1}{\sqrt{k}}\\
&\le C_1(\E|X_1|^3)^2\E|S_{\tau_x}|\frac{\sqrt{n}}{x+\sqrt{n}}.
\end{align*}
This completes the proof.
\end{proof}
%%%%%%%%%%%%%%%%%%%%%%%%%%%%%%%%%%%%%%%%%%%%%%%%%%%%%%%%%%%%%%%%%%%%%%%%%%%%%%%%%%%%%%%%%%%%%%%%%%%%%%%%%%%%%%%%%%%%%%%%%%%%%%%%%%%%%%%%%%%%%%%%%%%%%%%%%%%%%%%%%%
\section{Proof of Theorem~\ref{thm:BE}}
As in the previous section,  we shall always assume that $\sigma^2=1$.

The strategy of the proof of Theorem~\ref{thm:BE} is the same in the proof of the main result in \cite{DTW24a}. As in this paper, we shall use a smoothening with a random variable $U$ which has the density 
$$
g_A(x)= \frac{3}{\pi A} \left(\frac{1-\cos (Ax)}{Ax^2}\right)^2, 
\quad x\in \R,
$$
where  $A = (8\E |X_1|^3)^{-1}$.

The first step in our proof is a comparison of probabilities 
$\pr(x + S_n\ge y,\tau_x>n)$ and $\pr(x+S_n+U\ge y, \tau_x>n)$.

% Using the classical inequality ???, we conclude that, uniformly in $x$, 
% \begin{align}
% \nonumber 
%     \pr(S_n\in [y,y+z]) 
% &\le 
%     2\gamma_0\frac{\E|X_1|^{2+\delta}}{n^{\delta/2}}
% +
%     \frac{1}{\sqrt{2\pi}}
%     \int_{y/\sqrt n}^{(y+z)/\sqrt n} 
%         e^{-t^2/2}dt 
% \\\label{eq.conc}
% &\le 
%     \frac{\E|X_1|^{2+\delta}}{n^{\delta/2}} + \frac{z}{\sqrt{n}}, 
% \end{align}
% This implies that 
% \begin{align}
% \label{sn.tau}
% \nonumber
% &
%     \pr(x+S_n\in [y,y+z],\tau_x>n)\\
% \nonumber
% &\hspace{1cm}
% \le
%     \int_0^\infty
%         \pr(x+S_{\lfloor n/2\rfloor}\in dw,\tau_x>n/2)
%         \pr(S_{n-\lfloor n/2\rfloor}\in [y-w,y-w+z])\\
% &\hspace{1cm}
% \le
%     \left( 
%         \frac{\E|X_1|^{2+\delta}}{n^{\delta/2}} 
%     +
%         \frac{z}{\sqrt{n}}
%     \right)
%     \pr(\tau_x>n/2).
% \end{align}
% \begin{remark}
% The appearance of the third moment on the right hand side of \eqref{eq.conc} is the only reason for the third power of the Lyapunov ratio on the right hand side of \eqref{eq:BE1}. If one manages to replace $\E|X_1|^3$ by an absolute constant then one obtains \eqref{eq:BE1} with 
% $\frac{\E|X_1|^3}{\sigma^3}$ instead of its third power.
% \hfill$\diamond$
% \end{remark}

\begin{lemma}
\label{lem:smoothing}
For all $n\ge1$ and all $x,y\ge 0$ we have 
\begin{align}\label{eq:smooth.1}
% \nonumber
&
    \left|
        \pr(x+S_n+U \ge y, \tau_x >n )  
    -
        \pr(x+S_n \ge y,\tau_x>n)
    \right|
%     \\
% &\hspace{2cm}
\le 
    \frac{\gamma_1(\E |U|)}{\sqrt{n}}
    \pr(\tau_x>n/2)
% \frac{\gamma_1(\E|U|)}{\sqrt{n}}\pr(\tau>n/2)
\end{align}
and
\begin{equation}\label{eq:smooth.2}
    \pr(x+S_n+U \le -y, \tau_x >n )  
\le 
    \frac{\gamma_1(\E |U|)}{\sqrt{n}}
    \pr(\tau_x>n/2).
    % \frac{\gamma_1(\E|U|)}{\sqrt{n}}\pr(\tau>n/2). 
\end{equation}
\end{lemma}
The proof of this lemma is almost a verbatim repetition of Lemma~4 in~\cite{DTW24a}
and we give it just to be self-contained.
\begin{proof}
Using~\eqref{sn.tau}, we obtain  
\begin{align*}
&\int_0^{\infty} 
    \pr(U\in -dz) 
    \left|
        \pr(x+S_n-z \ge y, \tau_x >n )  
    -
        \pr(x+S_n \ge y,\tau_x>n)
    \right|\\
&\hspace{1cm}=
    \int_0^{\infty}
        \pr(U\in -dz) 
        \pr(x+S_n\in [y,y+z),\tau_x>n)\\
&\hspace{1cm}\le 
    \int_0^{\infty} 
        \pr(U\in -dz) 
        \frac{\sqrt 2 \E|X_1|^3+\pi^{-1/2}z }{\sqrt{n}}\pr(\tau_x>n/2)\\
&\hspace{1cm}= 
    \frac{\sqrt{2} \E|X_1|^3 \pr(U < 0)+\pi^{-1/2}\E U^- }{\sqrt{n}}
\pr(\tau_x>n/2)
\end{align*}
and
\begin{align*}
&\int_0^{\infty} 
    \pr(U\in dz) 
    \left|
        \pr(x+S_n+z \ge y, \tau_x >n )  
    -
        \pr(x+S_n\ge y,\tau_x>n)\right|\\
&\hspace{1cm}
=
    \int_0^{y} 
        \pr(U\in dz) 
        \pr(x+S_n\in [y-z,y),\tau_x>n)
        +
        \pr(U>y)
        \pr(x+S_n\in (0,y),\tau_x>n)\\
&\hspace{1cm}
\le 
    \int_0^{y} 
        \pr(U\in dz)
        \frac{\sqrt 2 \E|X_1|^3+\pi^{-1/2}z}{\sqrt{n}}
        \pr(\tau_x>n/2)\\
&\hspace{2cm}
+
    \pr(U>y)
    \frac{\sqrt 2 \E|X_1|^3+\pi^{-1/2} y}{\sqrt{n}}
    \pr(\tau_x>n/2)\\
&\hspace{1cm}
\le 
    \frac{\sqrt 2 \E|X_1|^3 \pr(U>0)+\pi^{-1/2} \E U^+}{\sqrt{n}}
    \pr(\tau_x>n/2).
\end{align*}
Combining these two inequalities we obtain~\eqref{eq:smooth.1}. 

The second claim follows again from \eqref{sn.tau}:
\begin{align*}
    &\pr(x+S_n+U \le -y, \tau_x >n)\\
&\hspace{2cm}=
    \int_y^\infty 
        \pr(U\in -dz)
        \pr(x+S_n\in(0, z-y], \tau_x >n)\\
&\hspace{2cm}\le 
    \int_y^\infty 
        \pr(U\in -dz)
        \frac{\sqrt 2 \E|X_1|^3+\pi^{-1/2} (z-y)}{\sqrt{n}}
        \pr(\tau_x>n/2)\\
&\hspace{2cm}\le 
    \frac{\sqrt 2 \E|X_1|^3 \pr(U<0)+\pi^{-1/2} \E U^-}{\sqrt{n}}
    \pr(\tau_x>n/2). 
\end{align*}
Thus, the proof of the lemma is complete.
\end{proof}
% \begin{remark}
% Since we know the exact formula for $\E|U|$, we have 
% $$
% \gamma_1(\E|U|)=\left(\frac{48\log2}{\pi^{3/2}}+\sqrt{2}\right)
% \E|X_1|^3.
% $$
% This absolute constant in front of $\E|X_1|^3$ is smaller than $7.39$.\hfill$\diamond$
% \end{remark}
By the total probability law,
\begin{align*}
    &\pr(x+S_n+U \ge y, \tau_x >n ) \\
&\hspace{1cm}=
    \pr(x+S_n+U \ge y) 
-
    \pr(x+ S_n+U \ge y,\tau_x\le n)\\
&\hspace{1cm}=
    \pr(S_n+U \ge y-x) \\
&\hspace{2cm}-
    \sum_{k=1}^{n}
    \int_{0}^\infty
        \pr(\tau_x = k, x+S_k \in -dz) \pr(S_{n-k}+U \ge y+z)\\
&\hspace{1cm}=
    \pr(S_n+U \ge y-x)
-
    \sum_{k=1}^{n}
        \pr(\tau_x = k) \pr(S_{n-k}+U\ge y)\\
&\hspace{2cm}+
    \sum_{k=1}^{n}
    \int_{0}^\infty
        \pr(\tau_x = k, x+S_k \in -dz) 
        \pr\big(S_{n-k}+U\in [y, y+z) \big)
\end{align*}
and
\begin{align*}
&\pr(x+S_n+U \le -y,\tau_x>n)\\
&\hspace{1cm}=
    \pr(x+S_n+U \le -y) 
-
    \pr(x+S_n+U \le -y, \tau_x \le n)\\
&\hspace{1cm}=
    \pr(S_n+U\le -y-x) \\
&\hspace{2cm}-
    \sum_{k=1}^{n}
    \int_{0}^\infty
        \pr(\tau_x = k, x+S_k \in -dz) \pr(S_{n-k}+U \le z-y)\\
&\hspace{1cm}=
    \pr(S_n+U \le -y-x)
-
    \sum_{k=1}^{n}
        \pr(\tau_x= k) \pr(S_{n-k}+U\le -y)\\
&\hspace{2cm}
-
    \sum_{k=1}^{n}
    \int_{0}^\infty
        \pr(\tau_x = k,x+S_k \in -dz ) 
                \pr(S_{n-k}+U \in(-y,-y + z])
       .
\end{align*}
Set
\begin{align}
\label{eq:Pn-def}
&P_n(x,y)\nonumber\\   
\nonumber
&:=
    \pr(x+S_n+U\ge y,\tau_x>n)
-
    \pr(x+S_n+U\le-y,\tau_x>n)\\
\nonumber
&=
    \pr(S_n+U \ge y-x)
-
    \pr(S_n+U \le -y-x)\\
&\hspace{0.5cm}-
    \sum_{k=1}^{n}
        \pr(\tau_x = k) 
        \big[ 
            \pr(S_{n-k}+U \ge y) - \pr(S_{n-k}+U \le -y )
        \big]\\
\nonumber        
&\hspace{0.5cm}+
    \sum_{k=1}^{n}
    \int_{0}^\infty
        \pr(\tau_x = k,x+ S_k\in -dz)
        \pr\left( S_{n-k} +U \in (-y, -y+z] \cup [y, y+z) \right).
\end{align}
It is immediate from Lemma~\ref{lem:smoothing} that 
\begin{equation}
\label{eq:smooth}
\sup_{y\ge0}
    \bigl|
        \pr(x+S_n\ge y,\tau_x>n)-P_n(x,y)
    \bigr|
\le
    2\frac{\gamma_1(\E |U|)}{\sqrt{n}}
    \pr(\tau_x>n/2).
\end{equation}

We now estimate the second half of the last sum in \eqref{eq:Pn-def}. 
Writing the convolution with $U$ as the integral, we have
\begin{align*}
    &\pr(S_{n-k}+U\in(-y,-y+z]\cup[y,y+z))\\
    &\hspace{1cm}= 
    \int_{-\infty}^\infty 
    \pr(U\in du) 
    \pr\left( S_{n-k} +u \in (-y,-y+z]\cup[y,y+z) \right)\\
    &\hspace{1cm}= 
    \int_{-\infty}^\infty 
    \pr(U\in du) 
    \pr\left( S_{n-k} \in (-y-u, -y-u+z] \cup [y-u, y-u+z) \right).
 \end{align*}
 Applying now \eqref{eq.conc}, we obtain
 \begin{align*}
\pr(S_{n-k}+U\in(-y,-y+z]\cup[y,y+z))
 \le \sqrt{2}\frac{\gamma_1(z)}{\sqrt{n-k}}. 
\end{align*}
This implies that
\begin{multline*} 
    \sum_{k=\lf n/2 \rf+1}^{n-1} 
    \int_{0}^\infty
        \pr(\tau_x = k, x+ S_k\in -dz)
        \pr\left( S_{n-k} +U \in (-y, -y+z] \cup [y, y+z) \right)
\\
\le
    \sqrt{2}
    \sum_{k=\lf n/2 \rf+1}^{n-1}
        \frac{1}{\sqrt{n-k}}
        \E[\gamma_1(|x+S_{\tau_x}|);\tau_x=k].
\end{multline*}
Applying now the third claim in Lemma~\ref{lem:S_tau}, we obtain
\begin{align}
\label{eq:sum_sesond_half}
\nonumber
&\sum_{k=\lf n/2 \rf+1}^{n-1} 
    \int_{0}^\infty
        \pr(\tau_x = k, x+S_k\in -dz)
        \pr\left( S_{n-k} +U \in (-y, -y+z] \cup [y, y+z) \right)\\
\nonumber
 &\hspace{1cm}\le \sqrt{2}C_1\E|S_{\tau_x}|(\E|X_1|^3)^3
   \sum_{k=\lf n/2 \rf+1}^{n-1} \frac{1}{k(x+\sqrt{k})(n-k)^{1/2}}\\
\nonumber 
 &\hspace{1cm}\le 4C_1\frac{\E|S_{\tau_x}|}{n(x+\sqrt{n})}(\E|X_1|^3)^3 
   \sum_{k=\lf n/2 \rf+1}^{n-1} \frac{1}{(n-k)^{1/2}}\\
   &\hspace{1cm}\le 8C_1\frac{\E|S_{\tau_x}|}{\sqrt{n}(x+\sqrt{n})}(\E|X_1|^3)^3.
\end{align}
To obtain an appropriate estimate for the sum over $k\le \lf n/2\rf$ we shall use the following estimate for the density $u \mapsto f_{S_n+U}(u)$ of the random variable $U+S_n$. Due to Lemma 9 in \cite{DTW24a}, uniformly in $u \in \R$,
\begin{align} \label{eq:normapprox}
\left|f_{S_n+U}(u) -\frac{1}{\sqrt{2\pi n}} e^{-\frac{u^2}{2n}}\right|\le 
\left(
\frac{72\E|X_1|^3}{\pi }
+\frac{\E|U|}{\sqrt{2\pi e}}
\right)
\frac{1}{n}
=:C_2\frac{\E|X_1|^3}{n}. 
\end{align}

Hence, for every $k\le n/2$,
\begin{multline*}
\Big|\pr\left( S_{n-k} +U \in (-y, -y+z] \cup [y, y+z) \right) \\
-
\frac{1}{\sqrt{2\pi (n-k)}} 
\int_{(-y, -y+z] \cup [y, y+z) }
e^{-\frac{u^2}{2(n-k)}}du 
\Big|
\le \frac{4C_2\E|X_1|^3}{n}z. 
\end{multline*}
Furthermore, by (32) in \cite{DTW24a} we have
\begin{align}
\label{eq:cont_of_normal}
\left|
    \frac{1}{\sqrt{n-k}} e^{-\frac{u^2}{2(n-k)}}
-  
    \frac{1}{\sqrt{n}} e^{-\frac{u^2}{2n}}
\right|
\le
    \frac{2^{3/2}}{e}\frac{k}{n^{3/2}}
\end{align}
uniformly in $u\in\R$ and in $k \le n/2$. 
Applying this bound, we get
\begin{multline*}
\left|\pr\left( S_{n-k} +U \in (-y, -y+z] \cup [y, y+z) \right) -
\frac{1}{\sqrt{2\pi n}} 
\int_{(-y, -y+z] \cup [y, y+z) }
e^{-\frac{u^2}{2n}}du 
\right|\\
\le \frac{4C_2\E|X_1|^3}{n}z+\frac{2^{3/2}}{e}\frac{kz}{n^{3/2}}. 
\end{multline*}
Upper bound (31) in \cite{DTW24a} implies that
\begin{align*}
\left|
   \int_{(-y, -y+z] \cup [y, y+z) }
e^{-\frac{u^2}{2n}}dz  - 2z 
e^{-\frac{y^{2}}{2n}}
\right|
\le
    \frac{2z^2}{e^{1/2}\sqrt{n}}
\end{align*}
and, consequently,
\begin{align}\label{eq:upper-bound1}
\nonumber
&\Bigg|
    \sum_{k=1}^{\lf n/2 \rf}
    \int_{0}^\infty
        \pr(x+S_k \in -dz, \tau_x = k)
        \pr\left( S_{n-k} +U \in (-y, -y+z] \cup [y, y+z) \right)\\
\nonumber
&\hspace{4cm}-
    \frac{2}{\sqrt{2\pi n}}
    e^{-\frac{y^2}{2n}}
    \sum_{k=1}^{\lf n/2 \rf}
        \int_0^\infty
        z\pr(x+S_k \in -dz , \tau_x = k)
\Bigg|\\
\nonumber
&\hspace{1cm}\le 
    \frac{\sqrt{2}}{\sqrt{e\pi}n} 
    \E[|x+S_{\tau_x}|^2;\tau_x\le n/2]
   + \frac{4C_2\E|X_1|^3}{n} \E[|x+S_{\tau_x}|;\tau_x\le n/2]   \\ 
   &\hspace{2cm}+
   \frac{2^{3/2}}{en^{3/2}}
   \sum_{k=1}^{\lf n/2 \rf}
    k\E[|x+S_{\tau_x}|;\tau_x=k].
\end{align}
Applying Lemmata~\ref{lem:S_tau-1}, \ref{lem:S_tau-2} and \ref{lem:S_tau-3}
to the corresponding terms on the right hand side of \eqref{eq:upper-bound1},
we obtain 
\begin{align}\label{eq:upper-bound2}
\nonumber
&\Bigg|
    \sum_{k=1}^{\lf n/2 \rf}
    \int_{0}^\infty
        \pr(x+S_k \in -dz, \tau_x = k)
        \pr\left( S_{n-k} +U \in (-y, -y+z] \cup [y, y+z) \right)\\
%\nonumber
&\hspace{0.5cm}-
    \frac{2}{\sqrt{2\pi n}}
    e^{-\frac{y^2}{2n}}
    \sum_{k=1}^{\lf n/2 \rf}
        \int_0^\infty
        z\pr(x+S_k \in -dz , \tau_x = k)
\Bigg|\le 
  \frac{C_3 \E |S_{\tau_x}| (\E|X_1|^3)^2}{\sqrt{n}(x+\sqrt{n})}.
\end{align}

Combining \eqref{eq:upper-bound2} and \eqref{eq:sum_sesond_half}, and applying the second inequality from Lemma~\ref{lem:S_tau} for all $k$ between $\lf n/2\rf$ and $n$
we conclude that
\begin{align}\label{eq:last-sum}
\nonumber
&\Bigg|\sum_{k=1}^{n}
    \int_{0}^\infty
        \pr(\tau_x = k, x+S_k\in -dz)
        \pr\left( S_{n-k} +U \in (-y, -y+z] \cup [y, y+z) \right)\\
      &\hspace{4cm}-
    \frac{2}{\sqrt{2\pi n}}
    e^{-\frac{y^2}{2n}}
    \E|x+S_{\tau_x}|
\Bigg|
%&\hspace{2cm}
\le \frac{C_4 \E |S_{\tau_x}| (\E|X_1|^3)^3}{\sqrt{n}(x+\sqrt{n})},
\end{align}
with some absolute constant $C_4$.

To estimate other terms in \eqref{eq:Pn-def} we first notice that, due to
the Berry-Esseen inequality~\eqref{eq.berry-esseen.classical}, 
we have, uniformly in $y\in\R$,
\begin{align*}
    \big|
        \pr(S_{n-k} \ge y) 
    -
        \pr(S_{n-k} \le -y)
    \big|
\le
    \gamma_0\frac{\E |X_1|^3}{\sqrt{n-k}}.
\end{align*}
Then, convolving with $U$ and taking into account the symmetry of $U$, we conclude that 
\begin{align}\label{eq:BE-diff}
   \sup_{y\in\R} \big|
        \pr(S_{n-k}+U \ge y) 
    -
        \pr(S_{n-k} +U \le -y)
    \big|
\le
    \gamma_0\frac{\E |X_1|^3}{\sqrt{n-k}}.
\end{align}
Combining this with the estimate
$\pr(\tau_x=k)
\le C_1\frac{\E|S_{\tau_x}|}{k(x+\sqrt{k})}(\E|X_1|^3)^2$ from Lemma~\ref{lem:S_tau}, we get
\begin{align}\label{eq:2nd-line-b}
\nonumber
    &\Bigg|
        \sum_{k=\lf n/2 \rf}^{n-1}
            \pr(\tau_x = k)
            \big[
                \pr(S_{n-k}+U \ge y) 
             -
                \pr(S_{n-k}+U \le -y)
            \big]
    \Bigg|\\
    \nonumber
&\hspace{2cm}\le
    2^{3/2}C_1\frac{\E|S_{\tau_x}|}{n(x+\sqrt{n})}(\E|X_1|^3)^3
    \sum_{k=\lf n/2 \rf}^{n-1}
        \frac{1}{\sqrt{n-k}}\\
&\hspace{2cm}\le
    2^{5/2}C_1\frac{\E|S_{\tau_x}|}{\sqrt{n}(x+\sqrt{n})}(\E|X_1|^3)^3.
\end{align}

Following \cite{DTW24a} we introduce
$$
Q_n(x):=\pr(S_n+U\ge x)-\pr(S_n+U\le-x).
$$
Then we have
\begin{align}\label{eq:sum-repr}
\nonumber
&\sum_{k=1}^{\lf n/2 \rf}
        \pr(\tau_x = k)
        \big[
            \pr(S_{n-k} +U\ge y) 
         -
            \pr(S_{n-k} +U\le -y)]\\
\nonumber
&\hspace{1cm}
 =\sum_{k=1}^{\lf n/2 \rf}
        \pr(\tau_x = k)Q_{n-k}(y)\\
 &\hspace{1cm}
 =Q_n(y)\pr(\tau_x\le \lf n/2 \rf)
   +\sum_{k=1}^{\lf n/2 \rf}
        \pr(\tau_x = k)[Q_{n-k}(y)-Q_n(y)].       
\end{align}
According to Lemma 11 in \cite{DTW24a}, 
\begin{align*}
    \sup_y\big|Q_{n-k}(y) - Q_{n}(y)\big| 
\le
    109\sqrt{\frac{3}{\pi}}2^{3/2}\E|X_1|^3\frac{k}{n^{3/2}}
\end{align*}
uniformly in $k\le n/2$. Combining this with Lemma~\ref{lem:tau-exp}, we infer that
\begin{align}\label{eq:Q-sum}
\nonumber
\left|\sum_{k=1}^{\lf n/2 \rf}
        \pr(\tau_x = k)[Q_{n-k}(y)-Q_n(y)]\right|    
 &\le C\E|X_1|^3\frac{1}{n^{3/2}}\sum_{k=1}^n k\pr(\tau_x=k)\\
  &\le C(\E|X_1|^3)^2\frac{\E|S_{\tau_x}|}{\sqrt{n}(x+\sqrt{n})}.
\end{align}
Furthermore, we know from \eqref{eq:BE-diff} that
$|Q_n(y)|\le \gamma_0\frac{\E|X_1|^3}{\sqrt{n}}$. Combining this with
\eqref{eq:tau2:corr0}, we have
\begin{equation}
\label{eq:Q-trivial}
\left|Q_n(y)\pr(\tau_x\le \lf n/2 \rf)-Q_n(y)\right|
\le 6\E|X_1|^3\frac{\E|S_{\tau_x}|}{\sqrt{n}(x+\sqrt{n})}.  
\end{equation}
Applying \eqref{eq:Q-sum} and \eqref{eq:Q-trivial} to the corresponding terms in \eqref{eq:sum-repr}, we conclude that 
\begin{align}\label{eq:2nd-line}
\nonumber
&\left|\sum_{k=1}^{\lf n/2 \rf}
        \pr(\tau_x = k)
        \big[
            \pr(S_{n-k} +U\ge y) 
         -
            \pr(S_{n-k} +U\le -y)]-Q_n(y)\right| \\
&\hspace{7cm}
    \le C(\E|X_1|^3)^2\frac{\E|S_{\tau_x}|}{\sqrt{n}(x+\sqrt{n})}.
\end{align}
Plugging \eqref{eq:last-sum}, \eqref{eq:2nd-line-b} and \eqref{eq:2nd-line}
into \eqref{eq:Pn-def}, we obtain 
\begin{align}
\label{eq:Pn-interm}
\nonumber
\Bigg|P_n(x,y)-\pr(S_n+U\ge y&-x)+\pr(S_n+U\le-y-x)+Q_n(y)\\
& -\frac{2}{\sqrt{2\pi n}} e^{-\frac{y^2}{2n}} \E|x+S_{\tau_x}|\Bigg|
 \le C(\E|X_1|^3)^3\frac{\E|S_{\tau_x}|}{\sqrt{n}(x+\sqrt{n})}.
\end{align}
We next notice that 
\begin{align}\label{eq:Q-repr}
\nonumber
&\pr(S_n+U\ge y-x)-\pr(S_n+U\le-y-x)-Q_n(y)\\
\nonumber
&\hspace{1cm}
=\pr(S_n+U\ge y-x)-\pr(S_n+U\ge y+x)+Q_n(x+y)-Q_n(y)\\
&\hspace{1cm}
=\pr(S_n+U\in[y-x,y+x))+Q_n(x+y)-Q_n(y).
\end{align}
Using \eqref{eq:normapprox} for $x\le\sqrt{n}$ and \eqref{eq:BE-diff} for $x>\sqrt{n}$, we conclude that 
\begin{align}
\label{eq:new1}
\nonumber
\left|Q_n(y+x)-Q_n(y)\right|
&=\left|\pr(S_n+U\in(-x-y,-y])-\pr(S_n+U\in[y,y+x)])\right|\\
&\le 2C_2\E|X_1|^3\frac{x}{\sqrt{n}(x+\sqrt{n})}
\end{align}
and, by similar arguments, 
\begin{align}
\label{eq:new2}  
\nonumber
&\left|\pr(S_n+U\in[y-x,y+x))
-\frac{1}{\sqrt{2\pi n}}
   \int_{y-x}^{y+x}e^{-u^2/2n}du\right|\\
&\hspace{2cm}\le 2C_2\E|X_1|^3\frac{x}{\sqrt{n}(x+\sqrt{n})}.
\end{align}
Applying \eqref{eq:new1} and \eqref{eq:new2} to the corresponding summands in \eqref{eq:Q-repr}, we conclude that 
\begin{align*}
&\left|
\pr(S_n+U\ge y-x)-\pr(S_n+U\le-y-x)-Q_n(y)
-\frac{1}{\sqrt{2\pi n}}
   \int_{y-x}^{y+x}e^{-u^2/2n}du \right|\\
  &\hspace{2cm} \le 4C_2\E|X_1|^3\frac{x}{\sqrt{n}(x+\sqrt{n})}.
\end{align*}
Combining this with \eqref{eq:Pn-interm}, we obtain 
\begin{align*}
&\left|P_n(x,y)
   -\frac{1}{\sqrt{2\pi n}}\int_{y-x}^{y+x}e^{-u^2/2n}du
   -\frac{2}{\sqrt{2\pi n}} e^{-\frac{y^2}{2n}} \E|x+S_{\tau_x}|\right|  \\
 &\hspace{2cm}  \le C(\E|X_1|^3)^3\frac{\E|S_{\tau_x}|}{\sqrt{n}(x+\sqrt{n})}.
\end{align*}
This bound, in combination with \eqref{eq:smooth} and \eqref{eq:tau2:corr0}, implies that 
\begin{align} \label{eq:mainresult}
&\left|\pr(x+S_n\ge y,\tau_x>n)
   -\frac{1}{\sqrt{2\pi n}}\int_{y-x}^{y+x}e^{-u^2/2n}du
   -\frac{2}{\sqrt{2\pi n}} e^{-\frac{y^2}{2n}} \E|x+S_{\tau_x}|\right|  \\
  &\hspace{4cm} \le \nonumber C(\E|X_1|^3)^3\frac{\E|S_{\tau_x}|}{\sqrt{n}(x+\sqrt{n})}.
 \end{align}
Thus, Theorem~\ref{thm:BE} is proved.

In order to prove Corollary~\ref{cor} we consider the function $\phi(t) = e^{-t^2/2}$ and use Taylor expansion. Then we have for $\theta = \theta(u) \in [0,1]$
\begin{align*}
    \phi((y+u)/\sqrt{n})
-
    \phi(y/\sqrt{n})
=
    u/\sqrt{n}\phi'(y/\sqrt{n})
+
    \frac{u^2}{2n}
    \phi''((y+\theta u)/\sqrt{n}).
\end{align*}
Integrating this from $-x$ to $x$ and multiplying by $1/\sqrt{2\pi n}$ we get
\begin{align*}
    \left|
        \frac{1}{\sqrt{2\pi n}}
        \int_{y-x}^{y+x}
            e^{-u^2/2n}du
    -
        \frac{2x}{\sqrt{2\pi n}}
        e^{-y^2/2n}
    \right|
&\le
    \sup\limits_{t \in \R}
    \phi''(t)
    \frac{1}{\sqrt{2\pi}n^{3/2}}
    \int_{-x}^x
        u^2/2du\\
&\le
    \frac{x^3}{3\sqrt{2\pi}n^{3/2}}.
\end{align*}
Combining this with \eqref{eq:mainresult} and noting that $x \le x+ \E|x + S_{\tau_x}| = \E |S_{\tau_x}|$ we obtain
\begin{align*}
    \left|
        \pr(x+S_n\ge y,\tau_x>n)
   -
        \frac{2}{\sqrt{2\pi n}} 
        e^{-\frac{y^2}{2n}} \E|S_{\tau_x}|\right|  
    \le C(\E|X_1|^3)^3\frac{\E|S_{\tau_x}|}{n}
+
    \frac{x^2\E|S_{\tau_x}|}{3\sqrt{2\pi}n^{3/2}}.
\end{align*}
Then considering $y=0$ we obtain
\begin{align*}
        \left|
        \frac{\pr(\tau_x>n)}
        {\sqrt{\frac{2}{\sigma^2\pi}}\E|S_{\tau_x}|n^{-1/2}}
    -
        1
    \right|
\le
    A_2\frac{(\E|X_1|^3)^3}{\sigma^9\sqrt{n}}
     + A_3\frac{x^2}{\sigma^2 n}
\end{align*}
for some absolute constants $A_2$ and $A_3$.
Hence 
\begin{align*}
        \left|
        \pr(x+S_n\ge y|\tau_x>n)
   -
        e^{-\frac{y^2}{2\sigma^2n}}
    \right|  
\le
    A_2\frac{(\E|X_1|^3)^3}{\sigma^9\sqrt{n}}
     + A_3\frac{x^2}{\sigma^2 n}.
\end{align*}
Thus Corollary~\ref{cor} is proved.

\section{Improvement of Theorem~\ref{thm:BE} for lattice and absolute continuous random walks}
As we have already mentioned in the introduction, the only reason for the appearance of the third power of the Lyapunov ratio in Theorem~\ref{thm:BE} is the bound \eqref{eq.conc}, where we have used the classical Berry-Esseen inequality to bound local probabilities for $S_n$. However,  local central limit theorems suggest that there should exist an upper bound for $\pr(S_n\in[y,y+z])$ which does not contain $\E|X_1|^3$ in the leading term. But to obtain such estimates one needs to introduce some assumptions on the 'local structure' of the distribution of increments $\{X_k\}$. In this section we show how such alternative estimates can be obtained in the case when the distribution of $X_1$ is either absolute continuous with respect to the Lebesgue measure or lattice with maximal span $1$.
Using these estimates we explain how to reduce the power of the Lyapunov ratio in Theorem~\ref{thm:BE}.

Let us start by mentioning a version of the Berry-Esseen inequality for local central limit theorem for densities. If the distribution of $X_1$ is absolutely continuous with a bounded density $p(x)$ then there exists an absolute constant $A$ such that, see the paper by Bobkov and Götze~\cite{BG25},
\begin{align}
\label{eq:BE-density}
\sup_{x\in\R}\left|p_n(x)-\frac{1}{\sqrt{2\pi}}e^{-x^2/2}\right|
\le A\frac{\E|X_1|^3}{\sqrt{n}}\|p\|_\infty^2,
\end{align}
where $p_n$ denotes the density of $S_n/\sqrt{n}$.

We next derive an analogue of this inequality for lattice random walks.
\begin{lemma}
\label{lem:BE-lattice}
Assume that all the conditions of Theorem~\ref{thm:BE} are valid. If the distribution of $X_1$ is lattice with the maximal span $1$ then
\begin{align*}
    \sup_{x \in \Z}
    \left|
        \sqrt{n} \pr(S_n = x) 
    -
        \frac{1}{\sqrt{2\pi}}
        e^{-x^2/2n}
    \right|
\le
    \left(
        \frac{76}{\pi}
    +
        \frac{24}{\pi V}
    \right)
    \frac{\beta_3}{\sqrt{n}},
\end{align*}
where 
\begin{align*}
    V = V(X_1) 
=
    -\sup\limits_{t \in (0,2\pi)}
    \frac{\log |\varphi(t)|}{1-\cos t}.
\end{align*}
\end{lemma}
\begin{remark}
The quantity $V(X_1)$ was introduced by Bobkov and Ulyanov in \cite{BU22} and can be seen as a quantitative characteristic of the assumption on the maximal span of $X_1$.
\hfill$\diamond$
\end{remark}
\begin{proof}[Proof of Lemma~\ref{lem:BE-lattice}]
By the inversion formula for lattice random variables,
\begin{align*}
    \pr(S_n = x)
= \frac{1}{2\pi}\int\limits_{-\pi}^{\pi} e^{-itx}\,\varphi^n(t)dt 
=
    \frac{1}{\sqrt{2\pi n}} \int\limits_{-\pi\sqrt n}^{\pi\sqrt n}
   e^{-it\,\frac{x}{\sqrt n}}\,
   \varphi^n\left(\frac{t}{\sqrt n}\right) dt .
\end{align*}

Consequently,
\begin{align*}
    \sqrt n\,\pr(S_n=x) - \frac{1}{\sqrt{2\pi}}e^{-x^2/2n}
= 
    \frac{1}{2\pi}\int\limits_{-\pi \sqrt{n}}^{-\pi \sqrt{n}}
   \varphi^n\!\left(\frac{t}{\sqrt n}\right)dt
   -\frac{1}{2\pi}
   \int\limits_{-\infty}^{\infty}e^{it\,\frac{x}{\sqrt n}}e^{-t^2/2} dt .
\end{align*}

Splitting the integrals at $\tfrac{1}{4L_n}=\tfrac{\sqrt n}{4\beta_3}$
(with $\beta_3=\mathbb{E}|X_1|^3$), we obtain
\begin{align} \label{eq:lattice:3integrals}
    &\sup_{x\in \Z} \Big|\sqrt{n}\pr(S_n=x) - \tfrac{1}{\sqrt{2\pi}}e^{-x^2/2n}\Big|\\
&\nonumber\hspace{4cm}
\le 
    \frac{1}{2\pi}\int\limits_{1/4L_n}^{1/4L_n}
   \Big|\varphi^n\!\Big(\tfrac{t}{\sqrt n}\Big) - e^{-t^2/2}\Big|dt \\
&\nonumber\hspace{4cm}
+
    \frac{1}{2\pi} \int\limits_{|t| \in [\frac{1}{4L_n}, \pi \sqrt{n}]} \big| \varphi^n(t/\sqrt{n})\big| dt
+
    \frac{1}{\pi} 
    \int\limits_{1/4L_n}^{\infty}
        e^{-t^2/2}dt\\
&\nonumber\hspace{4cm}
=:
    I_1 + I_2 + I_3 .
\end{align}
Using the bound
\[
\Big|\varphi^n\!\Big(\tfrac{t}{\sqrt n}\Big) - e^{-t^2/2}\Big|
\le 16L_n|t|^3 e^{-t^2/3}, \qquad |t|\le \tfrac{1}{4L_n},
\]
we obtain
\begin{align}\label{eq:lattice:I_1estim}
    I_1 \nonumber
&\le 
    \frac{16L_n}{\pi\sqrt n}\int_{0}^{1/(4L_n)} t^2 e^{-t^2/3}dt
\le
    \frac{16}{\pi}
    \frac{\beta_3}{\sqrt{n}}
    \int\limits_0^\infty
        t^3 e^{-t^2/3}dt\\
&=
    \frac{16}{\pi}
    \frac{\beta_3}{\sqrt{n}}
    \frac{9}{2}
    \int\limits_0^\infty
        z e^{-z}dz
=
    \frac{72}{\pi}
    \frac{\beta_3}{\sqrt{n}}.
\end{align}
For $I_3$ we have
\begin{align} \label{eq:lattice:T_3estim}
    I_3 
\le
    \frac{4}{\pi}
    \frac{\beta_3}{\sqrt{n}}
    \int\limits_{1/4L_n}^\infty
        te^{-t^2/2}dt
=
    \frac{4}{\pi}
    \frac{\beta_3}{\sqrt{n}}
    e^{-n/32\beta_3}
\le
    \frac{4}{\pi}
    \frac{\beta_3}{\sqrt{n}}.
\end{align}
In order to bound $I_2$ we notice that our assumption on the maximal span of $X_1$ implies that $V > 0$. Moreover, we have the bound
\begin{align*}
    \log |\varphi(t)| \le -V(1-\cos t), \quad t \in [-\pi, \pi].
\end{align*}
Therefore,
\begin{align*}
    I_2 
& \le 
    \frac{1}{\pi}
    \int\limits_{1/4L_n}^{\pi \sqrt{n}}
        e^{-nV(1-\cos (t/\sqrt{n}))}
        dt\\
& \le
    \frac{1}{\pi}
    \int_{1/4L_n}^{\pi \sqrt{n}}
        e^{-nV\left(\frac{t}{\sqrt{n}}\right)^2\left(1-\frac{\pi^2}{12}\right)}
        dt
\le
    \frac{1}{\pi}
    \int_{1/4L_n}^{\infty}
        e^{-\widetilde{V}t^2/2}
        dt,
\end{align*}
where $\widetilde{V} = V\left(1-\frac{\pi^2}{12}\right)$.
Then we have
\begin{align} \label{eq:lattice:I_2estim}
    I_2
\le
    \frac{4\beta_3}{\pi \sqrt{n}}
    \int\limits_{1/4L_n}^{\infty}
        te^{-\widetilde{V}t^2/2}
        dt
=
    \frac{4\beta_3}{\pi\widetilde{V} \sqrt{n}}
    e^{-n\widetilde{V}/32\beta_3}
\le
    \frac{4\beta_3}{\pi\widetilde{V} \sqrt{n}}.
\end{align}
Plugging \eqref{eq:lattice:I_1estim} - \eqref{eq:lattice:I_2estim} into \eqref{eq:lattice:3integrals}, we obtain 
\begin{align*}
    \sup_{x \in \Z}
    \left|
        \sqrt{n} \pr(S_n = x) 
    -
        \frac{1}{\sqrt{2\pi}}
        e^{-x^2/2n}
    \right|
&\le
    \left(
        \frac{76}{\pi}
    +
        \frac{4}{\pi (1-\pi^2/12)V}
    \right)
    \frac{\beta_3}{\sqrt{n}}\\
&\le
    \left(
        \frac{76}{\pi}
    +
        \frac{24}{\pi V}
    \right)
    \frac{\beta_3}{\sqrt{n}}.
\end{align*}
Thus, the proof of the lemma is complete.
\end{proof}
Using \eqref{eq:BE-density} in the absolutely continuous case and Lemma~\ref{lem:BE-lattice} in the lattice case we can sharpen the bound \eqref{eq.conc} for local probabilities. Indeed, \eqref{eq:BE-density} implies that 
\begin{equation*}
\pr(x+S_n\in[y,y+z])
\le\left(\frac{1}{\sqrt{2\pi n}}+ A\frac{\|p\|_\infty\E|X_1|^3}{n}\right)z.
\end{equation*}
Under the assumptions of Lemma~\ref{lem:BE-lattice} one has 
\begin{equation*}
\pr(x+S_n\in[y,y+z])
\le\left(\frac{1}{\sqrt{2\pi n}}+
\left(\frac{76}{\pi}+\frac{24}{\pi V}\right)
    \frac{\E|X_1|^3}{n}\right)(z+1).
\end{equation*}
Letting $z=1$ in these inequalities, we conclude that if the distribution of $X_1$ is absolute continuous or lattice with the maximal span $1$ then there exists an absolute constant $A$ such that 
\begin{equation}
\label{eq:new-conc}
\pr(x+S_n\in[y,y+1])
\le \frac{A}{\sqrt{2n}}\left(1+R\frac{|X_1|^3}{\sqrt{2n}}\right),
\end{equation}
where $R=\|p\|_\infty^2$ in the absolute continuous case and $R=V^{-1}$ in the lattice case.

Replacing \eqref{eq.conc} by \eqref{eq:new-conc} in the proof of Lemma~\ref{lem:concrete.bound.alternative}, we obtain
$$
\pr(x+S_n\in[y,y+1),\tau_x>n)
\le C\left(1+R\frac{|X_1|^3}{\sqrt{k}}\right)
\frac{\E|S_{\tau_x}|(y+\E|X_1|^3)}{n(x+\sqrt{n})}.
$$
This, in its turn, leads to the following improvements of bounds in Lemma~\ref{lem:S_tau}:
$$
\pr(\tau_x=k)\le C_1\left(1+R\frac{|X_1|^3}{\sqrt{k}}\right)
\frac{\E|S_{\tau_x}|}{k(x+\sqrt{k})}\E|X_1|^3,
$$
$$
\E[|x+S_{\tau_x}|;\tau_x=k]\le C_1\left(1+R\frac{|X_1|^3}{\sqrt{k}}\right)
\frac{\E|S_{\tau_x}|}{k(x+\sqrt{k})}\E|X_1|^3
$$
and
$$
\E[|x+S_{\tau_x}|^2;\tau_x=k]\le C_1\left(1+R\frac{|X_1|^3}{\sqrt{n}}\right)
\frac{\E|S_{\tau_x}|}{\sqrt{k}(x+\sqrt{k})}\E|X_1|^3
$$
for all $n\ge 32(\E|X_1|^3)^2$+5. 

Furthermore, using \eqref{eq:new-conc} in the arguments leading to \eqref{eq:sum_sesond_half}, we can see that the right hand side in this estimate changes to $C\frac{\E|S_{\tau_x}|}{\sqrt{n}(x+\sqrt{n})}\E|X_1|^3\left(1+R\frac{\E|X_1|^3}{\sqrt{n}}\right)$. Since this was the only place where the third power of $\E|X_1|^3$ arises, we conclude that
the right hand side in\eqref{eq:BE1} can replaced by
$$
C\frac{(\E|X_1|^3)^2\E|S_{\tau_x}|}{\sigma^6 \sqrt{n}(x+\sqrt{n})}
\left(1+R\frac{\E|X_1|^3}{\sqrt{n}}\right)
$$
provided that the distribution of $X_1$ is either absolute continuous or lattice with maximal span $1$.
\bibliographystyle{abbrv}
\bibliography{references}
\end{document}